\documentclass[a4paper,12pt]{article}
\usepackage{amsmath}
\usepackage{amssymb}
\usepackage{theorem}

\theorembodyfont{\itshape}
\newtheorem{thm}{Theorem}%[section]
\newtheorem{prop}[thm]{Proposition}
\newtheorem{lem}[thm]{Lemma}
\newtheorem{sublem}[thm]{Sublemma}
\theorembodyfont{\upshape}
\newtheorem{prob}[thm]{Problem}

\newtheorem{ex}[thm]{Example}
\newtheorem{claim}[thm]{Claim}

\makeatletter
\DeclareRobustCommand{\qed}{%
  \ifmmode % if math mode, assume display: omit penalty etc.
  \else \leavevmode\unskip\penalty9999 \hbox{}\nobreak\hfill
  \fi
  \quad\hbox{\qedsymbol}}
\newcommand{\openbox}{\leavevmode
  \hbox to.77778em{%
  \hfil\vrule
  \vbox to.675em{\hrule width.6em\vfil\hrule}%
  \vrule\hfil}}
\newcommand{\qedsymbol}{\openbox}
\newenvironment{proof}[1][\proofname]{\par
  \normalfont
  \topsep6\p@\@plus6\p@ \trivlist
  \item[\hskip\labelsep\itshape
    #1\@addpunct{.}]\ignorespaces
}{%
  \qed\endtrivlist
}
\newcommand{\proofname}{Proof}
\makeatother

\newcommand{\R}{\mathbb{R}}
\newcommand{\N}{\mathbb{N}}
\newcommand{\supp}{\mathop{\mathrm{supp}}\nolimits}

\renewcommand{\hat}{\widehat}
\renewcommand{\setminus}{\smallsetminus}

\renewcommand{\refname}%
{\begin{center}\normalsize\mdseries\scshape%
{References}\end{center}}

\title{On non-uniformly simple groups}
\author{Hiroki KODAMA}
\date{July 26, 2011}

\begin{document}
\maketitle
\thispagestyle{empty}

\begin{abstract}
Suppose $G$ is a simple group.
For any nontrivial elements $g$ and $h$,
$g$ can be written as a finite product of
conjugates of $h$ or the inverse of $h$.
G is called uniformly simple if the length
of such an expression is uniformly bounded.
We show that the infinite alternating group is
non-uniformly simple and evaluate how 
the length of such an expression is unbounded.
\end{abstract}

%\section{Introduction}
For an element $g$ of a group $G$,
we denote by $C_g$ the conjugacy class of $g$,
and define $[g] := C_g \cup C_{g^{-1}}$.
Suppose $G$ is a simple group.
For a nontrivial element $h\in G$,
the normal subgroup generated by $h$,
$$N_h=\{h_1h_2\cdots h_n \,|\, n\in\N,\, h_i \in [h]\}$$
coincides with $G$. Therefore, for any nontrivial elements $g,h\in G$,
$g$ can be written as a 
finite product of elements of $[h]$.
Denote by $\lambda_h (g)$ the minimum number of $n$ 
in such expressions $g=h_1h_2\cdots h_n$.
A simple group $G$ is called \textbf{uniformly simple} if $\lambda_h (g)$
is bounded by a constant independent of $g,h$ and
\textbf{non-uniformly simple} if it is not.

If $h_1,\,h_2 \in [h]$ and $g_1,\,g_2 \in [g]$ 
then $\lambda_{h_1} ({g_1})=\lambda_{h_2} ({g_2})$.
Therefore, we can consider $\lambda$ as a function of $[h]$ and $[g]$
to use an abuse of notation
$$
\lambda_{[h]} ([g])=\lambda_{[h]} (g)=
\lambda_{h} ([g])=\lambda_{h} (g).
$$

\begin{ex}
The group $S_\infty$ of bijections on the set $\N$ of 
natural numbers with finite support is called the 
\textbf{infinite symmetric group}. 
The even permutations of $S_\infty$ form the 
\textbf{infinite alternating group} $A_\infty$, 
which is a non-uniformly simple group.
\end{ex}

For any nontrivial elements $g,h \in G$, we set
$$d([g],[h]):=\log(\max\{\lambda_{[g]} ([h]),\lambda_{[h]} ([g])\}).$$
From an easy inequality 
$\lambda_{[f]} ([h])
\leqq
\lambda_{[f]} ([g])\lambda_{[g]} ([h])$ 
($[f],[g],[h]\in G\setminus\{e\}$) $\cdots$(1),
it follows that
$d$ is a metric on $\hat{G}:=\{[g]\,|\, g\in G\setminus\{e\} \, \}$.
Remark that the followings are equivalent; 
1.~the group $G$ is uniformly simple,
2.~the metric space $(\hat{G},d)$ is quasi-isometric to a point.

What if the group $G$ is non-uniformly simple?
For the case $G=A_\infty$ we show the following theorem.

\begin{thm}\label{th}
the metric space $(\hat{A_\infty},d)$
is quasi-isometric to the half-line $\R_{+}$
\end{thm}

In the rest of this paper we consider only 
the following two kind of groups, 
namely infinite symmetric group $S_\infty$ and 
the infinite alternating group $A_\infty$.
The following claim is important;

\begin{claim}
For any two elements $h,h' \in A_\infty$, 
they are conjugate in $A_\infty$ if and only if
they are conjugate in $S_\infty$.
\end{claim}

\begin{proof}
Suppose $h' = ghg^{-1}$ for an odd permutation $g \in S_\infty$.
Take two distinguished elements $a,b\in\N$ away from the support of $h'$.
Then $h' = ((a\,\,b)g)h((a\,\,b)g)^{-1}$ and $((a\,\,b)g)$ 
is an even permutation.
\end{proof}

We define $\lambda_h (g)$ similarly 
on the infinite symmetric group $S_\infty$
while $\lambda_h (g)=\infty$ if $g\not\in N_h$ 
i.e.\ $g$ is odd and $f$ is even.
Claim 3.\ assures that this definition is an extension of one on $A_\infty$.

Consider a transposition $\iota_1=(1\,\,2) \in S_\infty$.
For a permutation $g \in S_\infty$,
the number $\lambda_{\iota_1} (g)$ is called the 
\textbf{word length} of $g$ and
simply denoted by $\lambda(g)$.
$g$ is an even permutation if and only if 
$\lambda(g)$ is an even number.

Theorem 2 follows from the following evaluation.

\begin{prop}\label{prop}
For any nontrivial elements $g,h\in A_\infty$, the following inequality holds.
$$
\frac{\lambda(g)}{\lambda(h)}
\leqq
\lambda_h(g)
\leqq
4\frac{\lambda(g)}{\lambda(h)}+4.
$$
\end{prop}

The left hand side evaluation follows directly from
the inequality (1).

To acquire the upper evaluation, we show the following lemma.
We denote the product of k transposition $(1\,\,2)(3\,\,4)\cdots(2k-1\,\,2k)$
by $\iota_k$.

\begin{lem}\label{a}
For any nontrivial permutation $h$ in the infinite symmetric group $S_\infty$,
there exists an integer $\ell \geq \lambda(h)/4$ such that
$\iota_{2\ell} = (1\,\,2)(3\,\,4)\cdots(4\ell-1\,\,4\ell)$
can be written as a product of two permutations which are conjugate to $h$.
\end{lem}

\begin{proof}%[—ªØ]
The following table shows the lemma holds if $h$ is a cyclic permutation.
Any nontrivial permutation $h$ can be written 
as $h=h_1 \cdots h_n$ where $h_i$ are cyclic permutations and
$\lambda(h) = \lambda(h_1)+\cdots+\lambda(h_n)$, therefore the lemma follows.

\begin{center}
\tiny
\begin{tabular}{c|c|c} 
$\lambda(h)$ &   & $\ell$\\
\hline
$1$ & $(1\,\,2)^{-1} (3\,\,4) = \iota_2$ & $1$ \\
\hline
$2$ & $(1\,\,2\,\,3)^{-1} (1\,\,3\,\,4) = \iota_2$ & $1$ \\
$3$ & $(1\,\,2\,\,3\,\,5)^{-1} (1\,\,3\,\,4\,\,5) = \iota_2$ & $1$ \\
$4$ & $(1\,\,2\,\,3\,\,5\,\,6)^{-1} (1\,\,3\,\,4\,\,5\,\,6) = \iota_2$ & $1$ \\
\hline
$5$ & $(1\,\,2\,\,3\,\,5\,\,6\,\,7)^{-1} (1\,\,3\,\,4\,\,5\,\,7\,\,8) = \iota_4$ & $2$ \\
$6$ & $(1\,\,2\,\,3\,\,5\,\,6\,\,7\,9)^{-1} (1\,\,3\,\,4\,\,5\,\,7\,\,8\,\,9) = \iota_4$ & $2$ \\
$7$ & $(1\,\,2\,\,3\,\,5\,\,6\,\,7\,\,9\,\,10)^{-1} (1\,\,3\,\,4\,\,5\,\,7\,\,8\,\,9\,\,10) = \iota_4$ & $2$ \\
\hline
$8$ & $(1\,\,2\,\,3\,\,5\,\,6\,\,7\,\,9\,\,10\,\,11)^{-1} (1\,\,3\,\,4\,\,5\,\,7\,\,8\,\,9\,\,11\,\,12) = \iota_6$ & $3$ \\
$9$ & $(1\,\,2\,\,3\,\,5\,\,6\,\,7\,\,9\,\,10\,\,11\,\,13)^{-1} (1\,\,3\,\,4\,\,5\,\,7\,\,8\,\,9\,\,11\,\,12\,\,13) = \iota_6$ & $3$ \\
$10$ & $(1\,\,2\,\,3\,\,5\,\,6\,\,7\,\,9\,\,10\,\,11\,\,13\,\,14)^{-1} (1\,\,3\,\,4\,\,5\,\,7\,\,8\,\,9\,\,11\,\,12\,\,13\,\,14) = \iota_6$ & $3$ \\
\hline
$\vdots$ & $\vdots$ & $\vdots$ \\
\hline
$3k-1$ & $(1\,\,2\,\,3\,\,\cdots\,\,4k-3\,\,4k-2\,\,4k-1)^{-1} (1\,\,3\,\,4\,\,\cdots\,\,4k-3\,\,4k-1\,\,4k) = \iota_{2k}$ & $k$ \\
$3k  $ & $(1\,\,2\,\,3\,\,\cdots\,\,4k-3\,\,4k-2\,\,4k-1\,\,4k+1)^{-1} (1\,\,3\,\,4\,\,\cdots\,\,4k-3\,\,4k-1\,\,4k\,\,4k+1) = \iota_{2k}$ & $k$ \\
$3k+1$ & $(1\,\,2\,\,3\,\,\cdots\,\,4k-3\,\,4k-2\,\,4k-1\,\,4k+1\,\,4k+2)^{-1} (1\,\,3\,\,4\,\,\cdots\,\,4k-3\,\,4k-1\,\,4k\,\,4k+1\,\,4k+2) = \iota_{2k}$ & $k$ \\
\hline
$\vdots$ & $\vdots$ & $\vdots$ \\
\end{tabular} 
\end{center}
\end{proof}

\begin{lem}\label{b}
For $k=2n\ell$, 
$\iota_k$ can be written as a product of $n$ permutations 
which are conjugate to $\iota_{2\ell}$.
\end{lem}

\begin{proof}
Trivial.
\end{proof}

\begin{lem}\label{c}
For any even permutation $g \in A_\infty$ and
any integer $k\geqq\lambda(g)/2$,
$g$ can be written as a product of two permutations 
which are conjugate to $\iota_k$.
\end{lem}

\begin{proof}
Since $g$ can be written as a product of cyclic permutations, 
we consider the following sublemma.

\begin{sublem}
Any cyclic permutation with word length $2n$ 
can be written as a product of two permutations 
which are conjugate to $\iota_n$.
Any cyclic permutation with word length $2n+1$ 
can be written as a product of two permutations 
which are conjugate to $\iota_n$ and $\iota_{n+1}$ 
(or $\iota_{n+1}$ and $\iota_n$).
\end{sublem}

\begin{proof}[Proof for the sublemma.]

\begin{equation*}
\begin{split}
&\left( (2\,\,3)(4\,\,5)\cdots(2n\,\,2n+1) \right) \left( (1\,\,2)(3\,\,4)\cdots(2n-1\,\,2n) \right) \\
&\hspace{5cm} = (1\,\,3\,\, \cdots \,\,2n+1\,\,2n\,\,2n-2\,\,\cdots\,\,2),\\
&\left( (2\,\,3)(4\,\,5)\cdots(2n\,\,2n+1) \right) \left( (1\,\,2)(3\,\,4)\cdots(2n+1\,\,2n+2) \right) \\
&\hspace{5cm} = (1\,\,3\,\, \cdots \,\,2n+1\,\,2n+2\,\,2n\,\,\cdots\,\,2),\\
&\left( (1\,\,2)(3\,\,4)\cdots(2n+1\,\,2n+2) \right) \left( (2\,\,3)(4\,\,5)\cdots(2n\,\,2n+1) \right) \\
&\hspace{5cm} = (1\,\,2\,\,4\,\,\cdots \,\,2n+2\,\,2n+1\,\,2n-1\,\,\cdots\,\,3).\\
\end{split}
\end{equation*}
\end{proof}

Because of the sublemma, 
we can write $g=f_1 f_2$ where $f_1$ and $f_2$ are 
conjugate to $\iota_{\lambda(g)/2}$.
Take $a_1,\dots,a_{2r}$ away from $\supp g$, 
$$g= (f_1 (a_1\,\,a_2)\cdots(a_{2r-1}\,\,a_{2r})) 
     ((a_1\,\,a_2)\cdots(a_{2r-1}\,\,a_{2r}) f_2),$$
which shows the lemma \ref{c}.
\end{proof}

Proposition \ref{prop} follows from 
Lemmas \ref{a}, \ref{b} and \ref{c}.
We define the map
$\psi\colon \hat{A_\infty} \to \R_{+}$
by
$\psi([g])=\log \lambda([g])$,
then $\psi$ is quasi-isometric and Theorem \ref{th} is proved.

\begin{prob}
For other non-uniformly simple group $G$,
What is the shape of the metric space
$(\hat{G},d)$?
Is it quasi-isometric to the half-line?
\end{prob}
%\begin{thebibliography}{WW}
%\bibitem[]{}
%\end{thebibliography}

%Author
\par\noindent{\scshape \small
Graduate School of Mathematical Sciences,
University of Tokyo, \\3-8-1 Komaba,
Meguro-ku, Tokyo 153-9814, Japan.}
\par\noindent{\ttfamily kodama@ms.u-tokyo.ac.jp}

\end{document}